\documentclass[12pt,a4 paper]{article}
\usepackage{amssymb}
\usepackage{amsmath}
\usepackage{amsthm}
\usepackage{amsfonts}
\usepackage{amstext}

\newtheorem{Theorem}{Theorem}[section]
\newtheorem{Lemma}[Theorem]{Lemma}
\newtheorem{Proposition}[Theorem]{Proposition}
\newtheorem{Definition}[Theorem]{Definition}
\newtheorem{Corollary}[Theorem]{Corollary}

\newtheorem{Remark}[Theorem]{Remark}
\newcommand{\dif}{\mathrm{d}}
\newcommand{\di}{\mathrm{div}}
\newcommand{\p}{\partial}
\newcommand{\C}{\mathrm{curl}}

\def \inner#1#2{\langle\,#1, #2\,\rangle}

\begin{document}

\title{Existence  of weak solutions for  non-stationary flows of fluids with shear thinning dependent viscosities under slip boundary conditions in  half space}
\author{Aibin Zang}

\date{}

\maketitle

\begin{center}
The School of Mathematics and Computer Science, Yichun University, Yichun, Jiangxi, P.R.China, 336000\\

\footnotesize{Email: zangab05@126.com}

\end{center}

\begin{abstract}
The author treats  the system of motion for an incompressible non-Newtonian fluids of the stress tensor described by $p-$potential function subject to slip boundary conditions in $\mathbb{R}^3_+$. Making use of  the Oseen-type approximation  to this model and  the  $L^\infty$-truncation method, one can establish the existence theorem of weak solutions for $p-$potential flow with $p\in(\frac{8}{5},2]$ provided that large initial data. 
\end{abstract}

\section{Introduction}

Let $\Omega\subset\mathbb{R}^3$ be an open set.  For any $T<\infty$, set $Q_T=\Omega\times (0,T).$ The motion of a homogeneous, incompressible fluid through $\Omega$ is governed  by the following equations 
   \begin{equation}\label{1.1}
\left\{\begin{aligned} &\partial_t u-\di S+(u\cdot\nabla)u+\nabla \pi=f, &\mbox{in}~~Q_T \\[3mm]
&\nabla\cdot u = 0, &\mbox{in}~~Q_T \\[2mm]
&u|_{t=0}=u_0(x),& \mbox{in}~~\Omega,
\end{aligned}\right.
\end{equation}
where $u$ is the velocity,  $\pi$ is pressure and $f$ is the force, $u_0$ is initial velocity and $S=(s_{ij})_{i,j=1}^n$ is stress tensor. The above system (1.1) has to be completed by boundary conditions except that $\Omega$ is the whole space and by constitutive assumptions for the extra tensor.
It is well known that the extra tensors are characterized by Stoke's law $S=\nu D(u)$,  where $D(u)$ is the symmetric velocity  gradient, i.e. $$D_{ij}(u)=\frac{1}{2}\left(\frac{\p u_i}{\p x_j}+\frac{\p u_j}{\p x_i}\right).$$ Assume that $\nu$ is a constant, \eqref{1.1} is called incompressible Navier-Stokes equations.

However, there are phenomena that can  be described by  $\nu=\nu(|D(u)|)$ with power-law ansatz, which presents certain non-Newtonian behavior of the fluid flows. These models are frequently used in engineering literature. We can refer the book by Bird, Armstrong and Hassager \cite{BAH} and the survey paper due to M\'{a}lek and Rajagopal \cite{MAR}.  Typical  examples for the constitutive relations are 
\begin{equation}\label{1.2}
\begin{aligned}
S(D(u))=\mu(\delta +|D(u)|)^{q-2}D(u)\\
S(D(u))=\mu(\delta +|D(u)|^2)^{\frac{q-2}{2}}D(u),
\end{aligned}
\end{equation}
with $1<q<\infty, \delta\ge 0,$ and $\mu>0$.

 These models \eqref{1.2} were studied initially by Lady\v{z}henska(see \cite{LA01}, \cite{LA02} and\cite{LA03}) and Lions(see \cite{Lions}). There are many results  improved   in several directions by many mathematicians. For example, concerning  the steady problem, there are several results proving existence of weak solution and interior regularity in bounded domain, we can refer \cite{DMM,fr,FR01,AMG,NW} and references therein. The  results of regularity up to boundary for the Dirichlet problems  were  obtained by T. Shinlkin \cite{SH} and H. Beir\~{a}o da Veiga, {\it{et.al.} }\cite{b1,BE01,BE02,BE03,BE04,BE05,BE06,CR01, CR02}. 

For the unsteady Dirichlet boundary value problems, J. M\'{a}lek, J. Ne\u cas, and M. R\r u\u zi\u cka\cite{MNR} firstly obtained  the weak solution for $p\ge 2$. Later, L. Diening  {\it{et.al.} } had  generalized these results to  $p>\frac{8}{5}$ in \cite{DMM}  and $p>\frac{6}{5}$ in \cite{DRW}.  There are also many papers dealing with regularity of for evolution Dirichlet boundary problems and we refer instance to \cite{AMANN,BA,BP,BE03,BE04,BE05,BE06}.  For the study of  these models with space periodic boundary conditions, we refer to the monograph \cite {ma} and papers\cite{BDR,dien}.

It should be emphasized that theoretical contributions mostly concern the no-slip boundary conditions and space periodic boundary conditions. However, many other boundary conditions are important for engineer experiment and computation science. Commonly used boundary conditions are Navier-type boundary conditions, i.e.
 \begin{equation}\label{1.3} 
 \begin{aligned}
 u\cdot n=0,  (S\cdot n)_\tau-\alpha u_\tau=0,\mbox{on}~~\p\Omega\times(0,T),
 \end{aligned}
 \end{equation}
 where $\alpha$ is the smooth function
which were introduced by Navier in \cite{NA}.  Navier-stokes equations  under Navier slip boundary conditions were studied by many mathematician,\cite{B2, B3} and \cite{xin}. There also are  some results for non-Newtonian fluid. For example, in \cite{b1, EM}, the authors investigated the regularity of steady flows with  shear-dependent viscosity on the slip boundary conditions. M. Bul\'{\i}\v{c}ek, J. M\'{a}lek and K.R. Rajagopal obtained the weak solution for the evolutionary generalized Navier-stokes-like system of pressure and shear-dependent viscosity in  \cite{bu}. Later, in \cite{bu1} they treated the fluid dependent on pressure, shear-rate and temperature  with the Navier-type slip boundary conditions in the bounded domain.

 In this paper, we study the system \eqref{1.1}  with stress tensor $S$ induced by $p-$potential as in Definition 2.1 in the half space $\mathbb{R}^3_+$ with perfect slip boundary conditions on $x_3=0$, i.e.
 \begin{equation}\label{1.4}
 u_3|_{x_3=0}=\left.\frac{\p u_i}{\p x_3}\right|_{x_3=0}=0\,\,(i=1,2).
 \end{equation} 

There have been some existence results in unbounded to study. For instance, Galdi and Grisanti obtained the existence and regularity of 2D exterior problems for steady flows with shear-thinning Ladyzhenskaya model in \cite{ga}. M. Pokorn\'{y} in \cite{po} firstly obtained the weak solution  of Cauchy problem for $p-$potential  models approximated by the Bipolar non-Newtonian incompressible fluid. M. Bul\'{\i}\v{c}ek, J.M\'{a}lek, M. Majdoub and  J. M\'{a}lek studied the Cauchy problem of the fluid with pressure dependent viscosity in \cite{bu2}. They established the existence theorem for unsteady flows approximated by means of higher differentiability models with the power $p>\frac{9}{5}.$

 After  extend the initial velocity and force term by mirror method, you note that it is seem to transfer into the Cauchy problem as previous mentioned in \cite{po} and should have established the existence results. However, because of absent enough regularity and uniqueness, this Cauchy problem can not recover the system of \eqref{1.1} with the boundary conditions \eqref{1.4}. As the boundary of $\mathbb{R}^3_+$ is concerned, the methods mentioned both in \cite{po} are \cite{bu2} are not applied without suitable boundary conditions for higher derivatives of the velocity.

 We overcome these difficulties from the following two ways. Firstly,  we can regularize the convection term $u\cdot\nabla u$ in system \eqref{1.1}  and obtain an approximated system,
   \begin{equation}\label{1.5}
\left\{\begin{aligned} &\partial_t u-\di S(D(u))+(u^\epsilon\cdot\nabla)u+\nabla \pi=f, &\mbox{in}~~Q_T \\[3mm]
&\nabla\cdot u = 0, &\mbox{in}~~Q_T \\[2mm]
&u|_{t=0}=u_0(x),& \mbox{in}~~\Omega,
\end{aligned}\right.
\end{equation}
 subjects to the  slip boundary conditions \eqref{1.4}, where $u^\epsilon$ is regularization of $u$ in $W^{2,2}$. Using Galerkin method, one has got the existence of \eqref{1.5} in any semi-ball in $\mathbb{R}^3_+$. Letting the radius go to infinity, the Galerkin scheme converges to the solution of the approximated system. Second, we verify local Minty-Trick Theorem in \cite{wo} by means of the $L^\infty-$truncation method in \cite{bu}, and establish the existence of weak solution for $p-$potential fluids \eqref{1.1} with $p>\frac{8}{5}$ subject to the boundary conditions \eqref{1.4}.

The paper is organized as follows. In the next  section, after  introduce the notations and definitions, we will present some  auxiliary results and main theorem (see Theorem \ref{mresult}) in this paper.  In section 3, we set an approximated system and the approximated solution.  In section 4, we will  prove the main result in this paper  presenting  the existence of weak solution to the system \eqref{1.1} with the slip  boundary conditions \eqref{1.4} by $L^\infty-$truncation method. 
\section{Basic notations and statement of main result}
In this section, we will introduce function spaces and the definitions of $p-$potential function and  weak solution to system (1.1) with (1.5). We will also show the Korn-type inequalities for unbounded domain.
Let $M^{n\times n}$ be the vector space of all symmetric $n\times n$ matrices $\xi=(\xi_{ij})$. We equip $M^{n\times n}$ with scalar product $\xi:\eta=\sum_{i,j=1}^n \xi_{ij}\eta_{ij}$ and norm $|\xi|=(\xi:\eta)^{\frac{1}{2}}.$
\begin{Definition}
Let $p>1$ and let $F: \mathbb{R}_+\bigcup\{0\}\to\mathbb{R}_+\bigcup\{0\}$ be a convex function, which is $C^2$ on the $\mathbb{R}_+\bigcup\{0\}$, such that $F(0)=0,\,\ F'(0)=0.$ Assume that the induced function $\Phi:M^{n\times n}\to\mathbb{R}_+\bigcup\{0\}$, defined through $\Phi(B)=F(|B|),$ satisfies
\begin{alignat}{12}
&\sum_{jklm}(\p_{jk}\p_{lm}\Phi)(B)C_{jk}C_{lm}\ge\gamma_1(1+|B|^2)^{\frac{p-2}{2}}|C|^2,\label{2.1}\\
&|(\nabla_{n\times n}^2\Phi)(B)|\le\gamma_2(1+|B|^2)^{\frac{p-2}{2}}\label{2.2}
\end{alignat}
for all $B,C\in M^{n\times n}$ with constants $\gamma_1,\gamma_2>0.$ Such a function $F$, resp. $\Phi$, is called a $p-$\textbf{potential}.
\end{Definition}

We define the extra stress $S$ induced by $F$, resp. $\Phi,$ by
$$S(B)=\frac{\partial\Phi(B)}{\p B}=F'(|B|)\frac{B}{|B|}$$ for all $B\in M^{n\times n}\setminus\{ \mathbf{0}\}.$
From \eqref{2.1}, \eqref{2.2} and $F'(0)=0$, it is easy to know that $S$ can be continuously extended by $S(\mathbf{0})=\mathbf{0}$.

As in \cite{dien} and \cite{ma}, one can obtain from \eqref{2.1} and \eqref{2.2} the following properties of $S$.
\begin{Theorem}
There exist constants $c_1,c_2>0$ independent of $\gamma_1,\gamma_2$ such that for all $B,C\in M^{n\times n}$ there holds
\begin{alignat}{12}
&S(\mathbf{0})=\mathbf{0},\,\ \label{2.3}\\
&\sum_{i,j}S_{ij}(B)B_{ij}\ge c_1\gamma_1(1+|B|^2)^{\frac{p-2}{2}}|B|^2,\label{2.4}\\
&|S(B)|\le c_2\gamma_2(1+|B|^2)^{\frac{p-2}{2}}|B|.\label{2.5}
\end{alignat}
\end{Theorem}
Let $\Omega$ be an open set in $\mathbb{R}^3$ and $\phi:\mathbb{R}_+\to\mathbb{R}_+$ be an increasing continuous convex function that vanishes at zero. We define the Orlicz space  $L^\phi(\Omega)$ which is the Banach space  by
$$L^\phi(\Omega)=\left\{v~ \mbox{is a measurable function}: \int_{\Omega}\phi(|v|)\dif x<\infty\right\},$$
equipped with the norm
$$\|v\|_{L^\phi}=\inf\left\{\lambda>0:\int_{\Omega}\phi\left(\frac{|v|}{\lambda}\right)\dif x\le1\right\}.$$
The Sobolev-Orlicz space $W^{k,\phi}(\Omega)$ is also Banach space by
$$W^{k,\phi}(\Omega)=\{v\in L^\phi(\Omega):D^\alpha v\in L^\phi(\Omega),~ \forall \alpha,~ |\alpha|\le k\}$$ with the norm
$$\|v\|_{W^{k,\phi}}=\sum_{|\alpha|\le k}\|D^\alpha v\|_{L^\phi}.$$
Note that if $\phi(s)=s^p$ with some $p\in [1,\infty)$,  we write $L^p(\Omega)=L^\phi(\Omega)$  which corresponding to the standard one for Lebesgue space. $L^\infty(\Omega)$ and $W^{k,p}(\Omega)$ are defined by usual ways respectively. Let $\mathcal{D}(\Omega)$ be the set of functions $v(x)\in C_0^\infty(\Omega)$. we defined
$$D^{k,\phi}(\Omega)=\overline{\mathcal{D}(\Omega)}^{\|\cdot\|_{D^{k,\phi}}}~\mbox{with}~\|v\|_{D^{k,\phi}}=\|\nabla^k v\|_{L^\phi}.$$
Similarily, $D^{k,p}(\Omega)$ with $p\in [1,\infty)$ denotes the space $D^{k,\phi}(\Omega)$with $\phi(s)=s^p$. When $\Omega=\mathbb{R}^n_+$, set
\begin{equation*}
\begin{aligned}
&D^{1,\phi}_n(\mathbb{R}^3_+)=\{v\in D^{1,\phi}(\mathbb{R}^3_+):v(x)\to 
0~\mbox{as}~|x|\to\infty~\mbox{and}~v\cdot\hat{n}|_{x_3=0}=0\}\\
&H=\{v\in L^{2}(\mathbb{R}^3_+):\di v=0~\mbox{and}~v\cdot\hat{n}|_{x_3=0}=0\}\\
&H(\Omega)=\{v\in L^{2}(\Omega):\di v=0~\mbox{and}~v\cdot\hat{n}|_{\p\Omega}=0\}
\end{aligned}
\end{equation*}
In the rest of the paper, the functions $\phi,\psi:\mathbb{R}_+\to\mathbb{R}_+$ are defined as $\phi(s)=s^2(1+s^2)^{\frac{p-2}{2}}$ and $\psi=\min(s^6,s^{\frac{3p}{3-p}}).$

We list some  auxiliary lemmas needed in the following text. Some results need to deal with Orlicz spaces, we will frequently use the following results in Orlicz space.
\begin{Lemma}\label{l2.3}
	Let $p\in (1,3).$ There exists a constant $C>0$, such that 
	$$\|\mathbf{E} u\|_{D^{1,\phi}(\mathbb{R}^3)}\le C \|u\|_{D^{1,\phi}(\mathbb{R}^3_+)}$$
	for all $u\in D^{1,\phi}(\mathbb{R}^3_+)$, where $\mathbf{E}$ is the  extension operator. 
	\end{Lemma}
	\begin{proof}
		It is easy to obtain this conclusion by the idea in \cite{ad}.
	\end{proof}
\begin{Lemma}\label{l2.4}
		Let $p\in (1,3).$ There exists a constant $C>0$, such that 
		$$\|v\|_{L^\psi}\le C\|\nabla v\|_{L^\phi},$$
		for each $v\in D^{1,\phi}(\mathbb{R}^3_+)$.
	\end{Lemma}
\begin{proof}
	From Lemma \ref{l2.3} and Lemma 3.1 in \cite{bu2}, we can obtain this inequality.
	\end{proof}
	Combine  Lemma \ref{l2.3} and Lemma 3.2 in \cite{bu2}, one can conclude the following Korn's inequality.
\begin{Lemma}\label{l2.5}
	Let $1<p<2$. There is a constant $C>0$ such that for  all vector value functions $v\in D^{1,\phi}(\mathbb{R}^3_+)$ there holds
	$$\|\nabla v\|_{L^\phi}\le C\|D(v)\|_{L^\phi}.$$
	\end{Lemma}
	 As a consequence of Lemma \ref{l2.3}, Lemma \ref{l2.4} and Lemma \ref{l2.5}, it is easy to conclude the following interpolation result. 
\begin{Corollary}\label{c2.6}
	Assume that $q\in [2,\frac{5p}{3}]$ and $u$ satisfies
	$$\|u\|_{L^\infty(0,T;L^2(\mathbb{R}^3_+))}+\|D(u)\|_{L^p(0,T;L^\phi(\mathbb{R}^3_+))}\le K.$$
	Then there exists a constant $C$ independent of $u$ and $K$ such that
$$\|u\|_{L^q(0,T;L^q(\mathbb{R}^3_+))}\le CK.$$
\end{Corollary}

Next, we consider the following Laplace equation
\begin{equation}\label{Lequ}\left\{
\begin{aligned}
&\Delta u=f, \, \, \, &&\mbox{in}\,\, \mathbb{R}^3_+;\\
&\frac{\p u}{\p\hat{ n}}=0, \, \, \, &&\mbox{on}\,\, \{x_3=0\},\\
&u(x)\to 0, \,\, \, &&\mbox{as}\,\, |x|\to\infty
\end{aligned}\right.
\end{equation}
then there exists a unique solution given by  for any $x\in\mathbb{R}^3_+$
\begin{equation}\label{uform}
	u(x)=\int_{\mathbb{R}^3_+}\left(\frac{1}{|x-y|}-\frac{1}{|x-y^*|}
	\right)f\dif y
\end{equation}

with $y^*=(y_1,y_2,-y_3)$ for $y=(y_1,y_2,y_3)\in\mathbb{R}^3_+.$ 

Using the Calder\'{o}n-Zygmund singular integral operator theory, we conclude that for all $q\in (1,\infty)$
\begin{alignat}{12}
&\|\nabla^2 u\|_{L^q}\le K\|f\|_{L^q},\label{2.7}\\
&\|\nabla u\|_{L^q}\le K\|v\|_{L^q}~\mbox{with}~f=\di v,\label{2.8}\\
&\| u\|_{L^q}\le K\|S\|_{L^q}~\mbox{with}~f=\di\di S.\label{2.9}
\end{alignat}

For any given $R>0$, set $\Omega=\{x\in\mathbb{R}^3_+: |x|\le R\}$. Denote by $\p\Omega_1=\{x\in\mathbb{R}^3_+:|x|=R\}$ and $\p\Omega_2=\p\Omega\cap\{x_3=0\}.$ 
The following proposition shows that there exists a basis in $ W^{2,2}$ with the boundary conditions \eqref{1.4}.
\begin{Proposition} \label{pro2.7}The eigenvalue problem
	\begin{equation}\label{eigen2.8}
	\left\{\begin{aligned} &-\Delta u+\nabla p=\lambda u, &&\mbox{in}~~\Omega \\[3mm]
	&\nabla\cdot u = 0, &&\mbox{in}~~\Omega, \\[2mm]
	&u=0    &&\mbox{on}~~\p\Omega_1,\\
	&u_3=0,~~~\frac{\p u_1}{\p x_3}=\frac{\p u_2}{\p x_3}=0  &&\mbox{on}~~\p\Omega_2,
	\end{aligned}\right.
	\end{equation}
	$\lambda \in\mathbb{R},\,\, u\in D^{1,2}(\Omega)\cap H(\Omega) $ admits a denumberable positive eigenvalue $\{\lambda_i\}$ clustering at infinity. Moreover, the corresponding eigenfunctions $\{\omega_i\}$ are in $W^{2,2}(\Omega),$ and associate pressure fields $p_i\in W^{1,2}(\Omega)$. Finally, $\{\omega_i\}$ are orthogonal and complete in $W^{2,2}(\Omega)\cap H(\Omega).$ Furthermore, for any $\Omega'\subset\subset \Omega$,  $\omega_i\in W^{3,2}(\Omega'\cup \p\Omega_2).$
\end{Proposition}
\begin{proof}
	We can infer this conclusion from the estimates of Theorem 3.2 in \cite{mar} and the Rellich theorem. Then it is easy to obtain the estimates of regularity.
	\end{proof}

 Now, we state the  main result as follows. If $q\in (1,\frac{5p}{6}]$, let $U^*$ be the dual space of $$U=L^p(0,T;D^{1,\phi}(\mathbb{R}^3_+))\cap L^{q^\prime}(0,T;D^{1,q^\prime}(\mathbb{R}^3_+)).$$
\begin{Theorem}\label{mresult}
Let $\frac{8}{5}< p\le2,$ $\phi(s)=s^2(1+s^2)^{\frac{p-2}{2}}$ and $S(D(u))$ be induced by the $p-$potential of Definition  2.1. Let  $f\in L^2(0,T; H) $ or $f\in U^*$,  and $u_0\in H$ with $\nabla\cdot u_0=0$ in the sense of distribution. Then there exists  a vector function $$u\in L^\infty(0,T;H)\cap L^p(0,T; D^{1,\phi}(\mathbb{R}^3_+)) ~\mbox{with}~u_t\in U^* $$  satisfies the following identity
\begin{equation}\label{2.30}
\begin{aligned}
&\int_{0}^T \langle u_t,\psi\rangle\dif t+&\int_{Q_T}(S(x,t,D(u))-u\otimes u):D(\psi)\dif x\dif t=\int_{Q_T} f\cdot\psi \dif x\dif t
\end{aligned}
\end{equation}
holds for all $\psi\in U$ with $\di\psi=0,\,\, \psi_3|_{x_3=0}=0$, where $Q_T=\mathbb{R}^3_+\times (0,T).$
\end{Theorem}

\section{Construction of the approximation solutions}
In this section, for brevity, assume that $f=0$, we will set the approximation system and hunt for the solutions.
 Given $u\in L^2(\Omega)$, where $\Omega$ is an open proper subset  of $\mathbb{R}^3$ with smooth boundary, define $u^\epsilon$ by 
 \begin{equation}\label{v3.1}
 u^\epsilon=J_\epsilon u:= P \mathbf{E}^\prime\mathcal{F}_\epsilon \mathbf{E} u 
  \end{equation} 
  where the symbol $\mathbf{E}$ is an extension operator from $L^2(\Omega)$ to $L^2(\mathbb{R}^3)$, $\mathcal{F}_\epsilon$ is the standard mollifier operator, $\mathbf{E}^\prime$ is the restriction from $L^2(\mathbb{R}^3)$ to $L^2(\Omega)$, and  $P$ is the Leray projector from $L^2(\Omega)$ to $H(\Omega)$ \cite{kato}.
  Hence it is easy to see that 
  \begin{alignat}{12}
  &u^\epsilon\to u~\mbox{in}~ D^{1,\phi}(\Omega)\cap H(\Omega)~\mbox{as}~\epsilon\to 0,\label{v3.2}\\
  &\|u^\epsilon\|_{W^{2,2}}\le \frac{C}{\epsilon^2}\|u\|_{L^2} \label{v3.3}
  \end{alignat}
  where $C$ is independent of $\epsilon.$
  
  Now we investigate the following approximation system 
   \begin{equation}\label{v3.4}
   \left\{\begin{aligned} &\partial_t u-\di S(D(u))+(u^\epsilon\cdot\nabla)u+\nabla \pi=0, &\mbox{in}~~Q_T \\[3mm]
   &\nabla\cdot u = 0, &\mbox{in}~~Q_T \\[2mm]
   &u|_{t=0}=u_0(x),& \mbox{in}~~\mathbb{R}^3_+,\\
   &u_3|_{x_3=0}=\left.\frac{\p u_i}{\p x_3}\right|_{x_3=0}=0\,\,(i=1,2),\\
   &u(x)\to 0,~\mbox{as}~|x|\to\infty.
   \end{aligned}\right.
   \end{equation}
   where $u^\epsilon$ is given in \eqref{v3.1}.
\begin{Theorem}
Let $u_0\in  H$ with boundary conditions \eqref{1.4} in the distribution sense and $S(D(u))$ be induced by a $p-$potential from Definition 2.1 with $\frac{6}{5}<p<2$. Then there exists a weak solution $u\in L^p(0,T; D^{1,\phi}(\mathbb{R}^n_+))\cap L^\infty (0,T;H)$ to the  system \eqref{v3.4} in the following sense
\begin{equation}\label{v3.5}
\begin{aligned}
&\int_{0}^T \langle u_t,\psi\rangle\dif t+&\int_{Q_T}(S(x,t,D(u))-u^\epsilon\otimes u):D(\psi)\dif x\dif t=0
\end{aligned}
\end{equation}
holds for all $\psi\in U$ with $\di\phi=0,\,\, \psi_3|_{x_3=0}=0$, where $Q_T=\mathbb{R}^3_+\times (0,T).$
\end{Theorem}

\begin{proof}
The proof is based on standard arguments for approximation of  a solution by the Galerkin method. Let $\{\Omega_R\}$ be a sequence of  bounded domains expanding in $\mathbb{R}^3_+$,  such that $\Omega_R=\{x\in\mathbb{R}^3_+: |x|\le R\}$ with $R\to\infty$, and $$\mathbb{R}^3_+=\bigcup_{R=1}^\infty\Omega_R.$$

Now we consider the following problem in $\Omega_R$,

 \begin{equation}\label{appro3.4}
 \left\{\begin{aligned} &\partial_t u_R-\di S(D(u_R))+(u_R^\epsilon\cdot\nabla)u_R+\nabla \pi=0, &\mbox{in}~~\Omega_R\times [0,T] \\[3mm]
 &\nabla\cdot u_R = 0, &\mbox{in}~~\Omega_R\times [0,T] \\[2mm]
 &u_R|_{t=0}=u_{0,R}(x),& \mbox{in}~~\Omega_R,\\
 &\left.u_R\right|_{|x|=R}=u_{R,3}|_{x_3=0}=\left.\frac{\p u_{R,i}}{\p x_3}\right|_{x_3=0}=0\,\,(i=1,2).
 \end{aligned}\right.
 \end{equation}
 where both $u_R^\epsilon$ given by \eqref{v3.1} and $u_{0,R}=P \mathbf{E}^\prime u_0$ with respect to the domain $\Omega_R$.  For simplicity we omit the subscript $R$ for $u_R(x,t)$. In $\Omega_R$ we consider the eigenfunctions $\{\omega_i(x)\}_{i\in\mathbb{N}}$ of the problem \eqref{eigen2.8}, and put
 $$u_m(x,t)=\sum_{k=1}^{m} c_{k,m}(t)\omega_k(x).$$ 
 Denote $\omega_k^\epsilon(x)=J_\epsilon \omega_k(x)$ for all $k=1,2,\cdots,$ and one yields $c_{k,m}(t)$ solves the following system of ordinary differential equations
 \begin{equation}\label{ode4.1}
 \begin{aligned}
 ( \p_tu_m,\omega_k)+(S(D(u_m))-u_m^\epsilon\otimes u_m,D( \omega_k))=0,~k=1,\cdots, m.
 \end{aligned}
 \end{equation}
 Due to the continuity of $S$, the local-in-time existence  follows from Caratheodory theory.  The global-in-time existence will be established by the following {\it apriori} estimates.

Multiply the equations \eqref{ode4.1} by $c_{k,m}$, then sum over $k$ and integrate on $(0,t).$ It is easy to obtain
\begin{equation}\label{v4.2}
\begin{aligned}
\|u_m\|^2_{L^2(\Omega_R)}+2\int_0^t(S(D(u_m)),D (u_m))=\|u_0\|^2_{L^2(\Omega_R)}.
\end{aligned}
\end{equation}
Hence
\begin{equation}\label{v4.3}
\begin{aligned}
\sup_{0\le t\le T}\|u_m(t)\|^2_{L^2(\Omega_R)}+2\int_{0}^{t}\|u_m(\tau)\|^p_{D^{1,p}(\Omega_R)}\dif \tau\le C(R,u_0).
\end{aligned}
\end{equation} 

From \eqref{ode4.1}, \eqref{v4.3} and $\{\omega_i(x)\}_{i\in\mathbb{N}}$ is dense in $D^{1,p}(\Omega_R)\cap H(\Omega_R)$, one infers that 
\begin{equation}\label{es4.3}
\|u'_m\|_{L^{p'}(0,T;(D^{1,p}(\Omega_R)\cap H(\Omega_R))^*)}\le C(R,u_0).
\end{equation}
By Aubin-Lions Lemma, it follows, as $m\to\infty$
\begin{alignat}{12}
&u'_m\rightharpoonup u',\,\, &&~\mbox{weakly\,\, in}~ L^{p'}(0,T;(D^{1,p}(\Omega_R)\cap H(\Omega_R))^*)\label{v4.4};\\
&u_m \stackrel{*}{\rightharpoonup} u,\,\, &&~\mbox{weakly}^* \mbox{\,\, in}~ L^{\infty}(0,T;H(\Omega_R))\label{v4.5};\\
&u_m \rightharpoonup u,\,\,&&~\mbox{weakly\,\, in}~ L^p(0,T;D^{1,p}(\Omega_R)))\label{v4.6};\\
&u_m\longrightarrow u,\,\,&&~\mbox{strongly\,\, in}~ L^q(0,T;L^q(\Omega_R))\,~ q\in [1,\frac{5p}{3})\label{v4.7};\\
&S(D(u_m))\rightharpoonup \tilde{S}^R,\,\,&&~\mbox{weakly\,\, in}~ L^{p'}(\Omega_R\times(0,T)) \label{v4.8}.
\end{alignat}
It is easy to see that
\begin{equation}\label{v4.9}
\begin{aligned}
\inner{u_t}{\omega_k}+(\tilde{S}^R,D (\omega_k))-(u^\epsilon\cdot\nabla \omega_k,u)=0.
\end{aligned}
\end{equation}
Multiplying both sides of \eqref{v4.9} by $c_{k,m}$ and summing over $k$ we find
\begin{equation}\label{v4.10}
\begin{aligned}
\inner{u_t}{u_m}+(\tilde{S}^R,D (u_m))-(u^\epsilon\cdot\nabla u_m,u)=0.
\end{aligned}
\end{equation}
Let us pass to the limit for $m\to\infty$ in to \eqref{v4.10}.  By these properties \eqref{v4.4},  \eqref{v4.5} and \eqref{v4.8}, we know that as $m\to\infty$
\begin{eqnarray*}
	&&\int_0^T \inner{u_t}{u_m}\to \int_0^T\inner{u_t}{u}=\|u\|^2_{L^2(\Omega_R)}-\|u_0\|^2_{L^2(\Omega_R)},~\,\\
	&& \int_0^T(\tilde{S}^R,D (u_m))\to\int_0^T(\tilde{S}^R,D (u)).
\end{eqnarray*}
Since
\begin{equation*}
(u^\epsilon\cdot\nabla u_m,u)-(u^\epsilon\cdot\nabla u,u)=\int_{\Omega_R}(u^\epsilon\otimes u)\cdot\nabla(u_m-u)\dif x
\end{equation*}
From \eqref{v4.6} and \eqref{v4.7}, we know that $u^\epsilon\otimes u\in L^{2}(\Omega_R\times(0,T))$, Hence $\int_0^T(u^\epsilon\cdot\nabla u_m,u)\to 0$ as $m\to \infty$, since $(u^\epsilon\cdot\nabla u,u)=0.$
Subtracting \eqref{v4.10} by \eqref{ode4.1} and passing to limit as $m\to\infty$, we get
\begin{equation}\label{v4.11}
\lim_{m\to\infty}\int_0^T(S(D(u_m)),D(u_m))\dif t=\int_0^T(\tilde{S}^R,D(u))\dif t.
\end{equation}
Choose $\psi=1$ in $\Omega_R\times [0,T]$ in Appendix \cite{wo}, then by Minty-Trick, it shows that $$\tilde{S}^R=S(D(u))~\mbox{a.e.}~\mbox{in}~\Omega_R\times (0,T).$$

Next we must consider the limit as $R$ tends to $\infty$. Let $u^R$ be the Galerkin approximation solution for the equation \eqref{appro3.4}, therefore, from \eqref{v4.2}, we obtain
\begin{alignat}{3}
&\|u^R\|_{L^{\infty}(0,T;H)}+2c_1\gamma_1\int_{0}^{t}\int_{\Omega_R}\phi(u^R(x,\tau)\dif x\dif \tau\le  \|u_0\|_2^2:=K,\label{u4.12}
\end{alignat}
from above \eqref{u4.12}, there exists $C(K)$ independent of $R$ and $\epsilon$, we know that
\begin{equation}\label{d4.12}
\int_{0}^{T}\int_{|D(u)^R|\ge 1}|D(u^R)|^p\le C(K)~\mbox{and}~\int_{0}^{T}\int_{|D(u)^R|\le 1}|D(u^R)|^2\le C(K).
\end{equation}  
From \eqref{2.5},  we can observe that
\begin{equation}\label{s4.12}
\begin{aligned}
&\int_{0}^T\int_{\Omega_{R}}\left|S(D(u^R))\right|^r\le c_2\gamma_2 \int_{0}^T\int_{\Omega_{R}}\left(1+|D(u^R)|^2\right)^\frac{(p-2)r}{2}|D(u^R)|^r\\
&\le c_2\gamma_2\left(\int_{0}^T\int_{|D(u^R)|\ge 1}+\int_{0}^T\int_{|D(u^R)|\le 1}\right)\left(1+|D(u^R)|^2\right)^\frac{(p-2)r}{2}|D(u^R)|^r\\
&:=c_2\gamma_2(J_1+J_2).
\end{aligned}
\end{equation}
Since

\begin{alignat}{3}
&J_1=\int_{0}^T\int_{|D(u^R)|\ge 1}\left(1+|D(u^R)|^2\right)^\frac{(p-2)r}{2}|D(u^R)|^r\le \int_{0}^T\int_{|D(u^R)|\ge 1}|D(u^R)|^{(p-1)r}, \label{j11}\\
\mbox{or}~
&J_1=\int_{0}^T\int_{|D(u^R)|\ge 1}\left(1+|D(u^R)|^2\right)^\frac{(p-2)r}{2}|D(u^R)|^r\le \int_{0}^T\int_{|D(u^R)|\ge 1}|D(u^R)|^\frac{pr}{2},\label{j12}\\
&J_2=\int_{0}^T\int_{|D(u^R)|\le 1}\left(1+|D(u^R)|^2\right)^\frac{(p-2)r}{2}|D(u^R)|^r\le \int_{0}^T\int_{|D(u^R)|\le 1}|D(u^R)|^r\label{j2}
\end{alignat}
thus take $r=p'=\frac{p}{p-1}>2$ in \eqref{s4.12}, by \eqref{j11}, \eqref{j2} and $r=2$ in \eqref{s4.12}, by \eqref{j12}, \eqref{j2} respectively, we have 
\begin{equation}\label{S4.13}
\begin{aligned}
\int_{0}^T\int_{\Omega_{R}}\left|S(D(u^R))\right|^{p'}\dif x\dif t\le C(K),\\
	\int_{0}^T\int_{\Omega_{R}}\left|S(D(u^R))\right|^{2}\dif x\dif t\le C(K).
	\end{aligned}
\end{equation}
Hence, it follows that, for any $2\le s\le p'$, by the interpolation,  
\begin{equation}\label{s5.2}
\|S(D(u^R))\|_{L^{s}(Q_T)}\le C(K).
\end{equation}
One infers that
\begin{equation}\label{ut4.12}
\|(u^{R})^\prime\|_{L^{2}(0,T;(D^{1,2}(\Omega_R)\cap H(\Omega_R))^*)}\le C(\epsilon, K)
\end{equation}
We can extend the $u^R$ to zero outside of $\Omega_R$. From above estimates \eqref{u4.12}, \eqref{ut4.12}, \eqref{S4.13} and Aubin-Lions Lemma, we can extract a subsequence of $\{u^R\}$, denoting it by $\{u^{R_k}\}_{k\in\mathbb{N}}$, such that as $R_k\to\infty,$
\begin{alignat}{12}
&(u^{R_k})^\prime\rightharpoonup u',\,\, &&~\mbox{weakly\,\, in}~ L^{p'}(0,T;(D^{1,2}(\mathbb{R}^3_+)\cap H)^*)\label{va4.4};\\
&u^{R_k} \stackrel{*}{\rightharpoonup} u,\,\, &&~\mbox{weakly}^* \mbox{\,\, in}~ L^{\infty}(0,T;H)\label{va4.5};\\
&u^{R_k} \rightharpoonup u,\,\,&&~\mbox{weakly\,\, in}~ L^p(0,T;D^{1,\phi}((\mathbb{R}^3_+)))\label{va4.6};\\
&u^{R_k}\longrightarrow u,\,\,&&~\mbox{strongly\,\, in}~ L^q(0,T;L^q_{\rm loc}(\mathbb{R}^3_+))\,~ q\in [1,\frac{5p}{3})\label{va4.7};\\
&S(D(u^{R_k}))\rightharpoonup G,\,\,&&~\mbox{weakly\,\, in}~ L^{p'}(\mathbb{R}^3_+\times(0,T)) \label{va4.8}.
\end{alignat}
  It suffices to study the more estimates for $u^R$ such that $D(u^{R_k})\to D(u) $ a.e. in $\mathbb{R}^3_+\times(0,T)$, as $R_k\to \infty.$  Since we can not obtain more regularity estimates from the above estimates, thus we start at the Galerkin approximation solutions $\{u^{R}_m\}$    for  given domain $\Omega_{R}$, and still denote it by $u_{m}$ as above.
  
 From Proposition \ref{pro2.7}, let  us consider $\eta$ smooth cut-off functions such that
  $\eta(x)=1$ for $|x|\le \frac{R}{2}$ and $\eta(x)=0$ for $|x|\ge R$, for any $x\in\mathbb{R}^3$, $|\nabla^k\eta|\le \frac{C}{R^k}$ and $\left|\frac{\nabla \eta}{\eta}\right|\le \frac{C}{R}.$ We observe that 
  \begin{equation}\label{bdry4.9}
  \begin{aligned}
  \C (\eta^2\C u_m)=2\eta\nabla\eta\times\C u_m+\eta^2\Delta u_m,~\C (\eta^2\C u_m|_{|x|=R})=0,\\
   \C (\eta^2\C u_m)\cdot\hat{n}|_{x_3=0}=\eta\nabla\eta\cdot(\C u_m\times\hat{n})|_{x_3=0}+\eta^2\Delta u_{m,3}=0.
  \end{aligned}
  \end{equation}
Therefore, take $-\C (\eta^2\C u_m)$  as a test function in \eqref{ode4.1}, it infers that 
\begin{equation}\label{ode4.10}
\begin{aligned}
( \p_tu_m,-\C (\eta^2\C u_m))+(S(D(u_m))-u_m^\epsilon\otimes u_m,D( -\C (\eta^2\C u_m)))=0.
\end{aligned}
\end{equation}
We will give the estimates one by one, $( \p_tu_m,-\C (\eta^2\C u_m))=\frac{1}{2}\|\eta\C u_m\|^2_2$. From the boundary conditions \eqref{1.4} and integrate by parts, we have 
\begin{equation}\label{s4.11}
	\begin{aligned}
&|(u_m^\epsilon\otimes u_m,D( \C (\eta^2\C u_m)))|=\left|\int_{\Omega_R}\C(u_m^\epsilon\cdot\nabla u_m)\cdot \eta^2\C u_m\dif x\right|\\
&=\left|\int_{\Omega_R}\eta^2\C u_m^\epsilon\cdot\nabla u_m\cdot\C u_m\dif x+\int_{\Omega_R}\eta^2u_m^\epsilon\cdot\nabla \C u_m\cdot\C u_m\dif x\right|\\
&\le \|\nabla u^\epsilon_m\|_{6}\|\eta\nabla u_m\|_3\|\eta\C u_m\|_2+\left\|\frac{\nabla\eta}{\eta}\right\|_\infty\|u^\epsilon_m\|_\infty\|\eta \C u_m\|_2^2\\
&\le\frac{1}{\epsilon^2}\|u_0\|_2\|\eta\nabla u_m\|_3\|\eta\C u_m\|_2+\frac{C}{R\epsilon^2}\|u_0\|_2\|\eta \C u_m\|_2^2.
	\end{aligned}
\end{equation}
Since $S(D(u))$ is deduced by a $p-$potential $\Phi$ and satisfies the property \eqref{2.1}, hence
\begin{equation}\label{es4.12}
	\begin{aligned}
&(S(D(u_m)),D(- \C (\eta^2\C u_m)))
=-\int_{\Omega_R}(S(D(u_m)): \C (\eta^2\C D(u_m)))\dif x\\
&-\int_{\Omega_R}(S(D(u_m)): \C (2{\rm sym}(\eta\nabla\eta\otimes\C u_m)))\dif x\\
&=\int_{\Omega_R}(\eta^2\C S(D(u_m)): \C D(u_m))\dif x-\int_{x_3=0}(\eta^2S(D(u_m))\times\hat{n}:\C D(u_m))\dif\sigma\\
&-\int_{\Omega_R}(S(D(u_m)): \C (2{\rm sym}(\eta\nabla\eta\otimes\C u_m)))\dif x\\
&\ge c_1\gamma_1\int_{\Omega_R}\eta^2(1+|D(u_m)|^2)^\frac{p-2}{2}\left|\C D(u_m)\right|^2\dif x+I_1+I_2,
	\end{aligned}
\end{equation}
where $I_1=-\int_{x_3=0}(\eta^2S(D(u_m))\times\hat{n}:\C D(u_m))\dif\sigma$, $I_2=-\int_{\Omega_R}(S(D(u_m)): \C (2{\rm sym}(\eta\nabla\eta\otimes\C u_m)))\dif x$,
hereafter the symbol ${\rm sym} (B)$ is the symmetric part of the matrix $B$, and denote $I_\eta(u)=\int_{\Omega_R}\eta^2(1+|D(u)|^2)^\frac{p-2}{2}\left|\C D(u)\right|^2\dif x.$

Since  the boundary conditions \eqref{1.4} and incompressibility imply that $$D_{jk}(u_m)\p_3D_{jk}(u_m)=0$$ for $j,k=1,2,3$ on $x_3=0$, thus from the formula $S(D(u))=F'(|D(u)|)\frac{D(u)}{|D(u)|}$, we have
\begin{equation}\label{s4.13}
\begin{aligned}
&I_1=-\int_{x_3=0}\eta^2(S(D(u_m))\times\hat{n}:\C D(u_m))\dif\sigma\\
&=\sum_{i,j,k}\int_{x_3=0}\eta^2F'(|D(u)|)\frac{D_{jk}(u)}{|D(u)|}n_i\varepsilon_{ijs}\p_i D_{jk}(u_m)\varepsilon_{ijs}\dif\sigma\\
&=\sum_{i,j,k}\int_{x_3=0}\eta^2F'(|D(u)|)\frac{D_{jk}(u)}{|D(u)|}\p_3 D_{jk}(u_m)=0
\end{aligned}
\end{equation}
and from the construction of cut-off function $\eta$, one obtains 
\begin{equation}\label{s4.14}
\begin{aligned}
&|I_2|\le \left|\int_{\Omega_R}(S(D(u_m)): \C (2{\rm sym}(\eta\nabla\eta)\otimes\C u_m))\dif x\right|\\
&+\left|\int_{\Omega_R}(S(D(u_m)): 2{\rm sym}(\eta\nabla\eta)\otimes\Delta u_m)\dif x\right|\\
&\le\frac{C}{R^2}\int_{\Omega_R}\left(1+|D(u_m)|^2\right)^\frac{p-power2}{2}|D(u_m)|^2\dif x+\frac{C}{R}\int_{\Omega}\eta\left|S(D(u_m))\Delta u_m\right|\\
&\le\frac{C}{R^2}+\frac{1}{R}\int_{\Omega_R}\eta\left(1+|D(u_m)|^2\right)^\frac{p-2}{2}|D(u_m)||\Delta u_m|\\
&\le \frac{C}{R^2}+\frac{1}{R}I_\eta(u_m)^\frac{1}{2}\left(\int_{\Omega_R}\left(1+|D(u_m)|^2\right)^\frac{p-2}{2}|D(u_m)|^2\dif x\right)^\frac{1}{2}\le \frac{C(\delta)}{R^2}+\delta I_\eta(u_m)
\end{aligned}
\end{equation}
where the constants $C$ are independent of $R,m$.
The following three interpolation inequalities will be used 
\begin{alignat}{12}	
&\|\cdot\|_3\le\|\cdot\|_2^\alpha\|\cdot\|_{3p}^{1-\alpha}~\alpha=\frac{2(p-1)}{3p-2},\label{neq4.15}\\
&\|\cdot\|_3\le\|\cdot\|_p^\beta\|\cdot\|_{3p}^{1-\beta}~\beta=\frac{p-1}{2}, \label{neq4.16}\\
&I_\eta(u)\ge C\|D(u)\|_{3p,\ge}^p, \label{neq4.17}
\end{alignat}
where $\|v\|^s_{s,\le}:=\int_{|v|\le 1}|v|^s\dif x$ and $\|v\|^s_{s,\ge}:=\int_{|v|\ge 1}|v|^s\dif x$.  Therefore, from \eqref{neq4.15}, \eqref{neq4.16} and \eqref{neq4.17}, we can follow
\begin{equation}\label{v4.18}
\begin{aligned}
&\|\eta\nabla u_m\|_3\|\eta\C u_m\|_2\le C\left(\|\nabla u_m\|_{2,\le}^\frac{2}{3}+ I_\eta(u_m)^\frac{1-\beta}{p}\|\nabla u_m\|_{p,\ge}^\beta\right)\|\eta\C u_m\|_2\\
&\le C\|\nabla u_m\|_{2,\le}^\frac{2}{3}\|\eta\C u_m\|_2+\|\nabla u_m\|^\frac{p}{3}_{p,\ge}\|\eta\C u_m\|_2^\frac{2p}{3(p-1)}+\delta I_\eta(u_m)\\
&\le C(\|\nabla u_m\|_{2,\le}^2+\|\nabla u_m\|_{p,\ge}^p)^\frac{1}{3}(\|\eta\C u_m\|_2^2+1)^\lambda+\delta I_\eta(u_m) 
\end{aligned}	
\end{equation}
where $\lambda=\frac{2p}{3p-3}>1$ with $p\in (\frac{8}{5},2] $ and $C$ is independent of $R,m$.

Combining \eqref{ode4.10}, \eqref{s4.11}, \eqref{s4.13}, \eqref{s4.14} and \eqref{v4.18}, we infer
\begin{equation*}
\begin{aligned}
\frac{\dif}{\dif t}\|\eta\C u_m\|_2^2+ CI_\eta(u_m)\le C(\|\nabla u_m\|_{2,\le}^2+\|\nabla u_m\|_{p,\ge}^p+\frac{1}{R^2})(\|\eta\C u_m\|_2^2+1)^\lambda
\end{aligned}
\end{equation*}
with $\lambda=\frac{2p}{3p-3}$, where $C$ only depends on $\epsilon$ and $\|u_0\|_2.$
Dividing above inequality by $(\|\eta\C u_m\|_2^2+1)^\lambda$, integrating with respect to time and using the estimates \eqref{d4.12}, we deduce that there is a constant $C$ that depends on $\epsilon$ and $\|u_0\|_2^2$, does not depend on $m,R$ such that
\begin{equation}\label{v4.19}
	\int_{0}^{T}I_\eta(u_m)(\|\eta\C u_m\|_2^2+1)^{-\lambda}\dif t\le C\left(1+\frac{1}{R^2}\right).
\end{equation}
Subtly modify the proofs of Lemma 6.2 and Lemma 6.3  in \cite{bu2}, we can conclude that for any bounded domain $\Omega^\prime$ of $\mathbb{R}^3_+$, as $R$ large enough, there exist a constant $C$ independent of $R,m$ and $0<\gamma<1 $ such that 
\begin{equation}\label{v4.20}
	\int_{0}^T\left(\int_{\Omega^\prime}|\nabla^2 u_m|^p\dif x\right)^\gamma\dif t\le C.
\end{equation}

 Applying the standard interpolation, \eqref{u4.12}, \eqref{v4.20} and the Young inequality, we conclude that there exists a constant $C$ is independent of $R,m$ such that
 $$	\int_{0}^T\|u_m^R\|^p_{W^{1+\sigma,p}(\Omega')}\le \int_{0}^{T}\|u_m^R\|^{p(1-\sigma)}_{W^{1,p}(\Omega')}\|u_m^R\|^{p\sigma}_{W^{2,p}(\Omega')}\dif t\le C$$provided that $\sigma$ small enough. Whence let $m\to \infty$ we have got 
 $$	\int_{0}^T\|u^{R}\|^p_{W^{1+\sigma,p}(\Omega')}\le C$$
 where $C$ does not depend on $R$. Since the compact embedding $W^{1+\sigma,p}(\Omega')$ into $W^{1,p}(\Omega')$,  Thus we  obtain there exists a subsequence $\{u^{R_k}\}$ such that 
 $$D(u^{R_k})\to D(u)~\mbox{a.e. in}~  \mathbb{R}^3_+\times (0,T), ~\mbox{as}~R_k\to\infty, $$
 which finishes the proof.
\end{proof}
\section{The proof of main result}
We have constructed the approximation solution $u_\epsilon$ for the problem \eqref{1.1}, and now we will show the approximation solutions  converge to the solution of the original problem \eqref{1.1} in this section. Denote $Q_T=\mathbb{R}^3_+\times (0,T).$

To finish the proof of theorem \ref{mresult}, we need to consider the estimates of  the pressure as follows.
\begin{Proposition}
Let $\frac{8}{5}< p\le2,$ $\phi(s)=s^2(1+s^2)^{\frac{p-2}{2}}$ and $S(D(u))$ be induced by the $p-$potential of Definition  2.1.  Assume that  a vector function $$u\in L^\infty(0,T;H)\cap L^p(0,T; D^{1,\phi}(\mathbb{R}^3_+)) ~\mbox{with}~u_t\in U^* $$  satisfies the following identity \eqref{v3.5} for any 
$\psi\in U$ with divergence free and $\psi_3|_{x_3=0}=0.$ Then there exists a unique pressure $\pi$ solves the  following problem 
\begin{equation}\label{pi}
\int_{0}^T \langle u_t,\varphi\rangle\dif t+\int_{Q_T}(S(D(u))-u^\epsilon\otimes u:D(\varphi))+\pi\nabla\cdot\varphi\dif x\dif t=0
\end{equation}
 for any $\varphi\in C^\infty(Q_T) $ with $\varphi_3|_{x_3=0}=0.$ Here $\pi=\pi_1+\pi_2$ with 
 \begin{equation}\label{pre}
  \|\pi_1\|_{\frac{r}{2}}\le C(r)\|u^\epsilon\otimes u\|_{r},~\forall r\in (2,\frac{5p}{3}],~ \|\pi_2\|_s\le C(s)\|D(u)\|_{L^\phi}, ~\forall s\in [2,p'].
 \end{equation}
 Moreover, if there is a sequence $\{u_n\}_{n=1}^\infty$ satisfying
 \begin{alignat*}{12}
 &u_n \rightharpoonup u~ \mbox{weakly in}~ L^r(Q_T)~ \mbox{for all}~ r\in [2,\frac{5p}{3}],\\
 &u_n\to u ~a.e. ~\mbox{in}~Q_T,\\
 &\nabla u_n\to\nabla u~ a.e. ~\mbox{in}~Q_T,\\
 &\|u_n\|_{D^{1,\phi}}\le C ~\mbox{uniformly w.r.t.}~ n,
 \end{alignat*}
 then there is a subsequence $\{\pi^n\}$ solving \eqref{pi} such that ($\pi^n=\pi^n_1+\pi^n_2$)
 \begin{alignat*}{12}
 \pi^n_1 \rightharpoonup \pi_1 ~ \mbox{weakly in}~ L^q(Q_T)~ \mbox{for all}~ q\in (1,\frac{5p}{6}],\\
  \pi^n_2 \rightharpoonup \pi_2 ~ \mbox{weakly in}~ L^s(Q_T)~ \mbox{for all} ~s\in [2,p'],
 \end{alignat*}
 and the pair $(u,\pi)$ solves \eqref{pi}.
	\end{Proposition} 
\begin{proof}
	The proof follows almost the procedure showed in \cite{bu2}, therefore, we can check the result for the smooth vector functions $u\in \mathcal{D}(\mathbb{R}_+^3)$ with divergence free and the boundary conditions \eqref{1.4}. Hence, we solve the pressure $\pi=\pi_1+\pi_2$ to the problem \eqref{pi} determined by the following two equations
	\begin{equation}\label{pi1}\left\{
	\begin{aligned}
	&-\Delta \pi_1=\di\di (u^\epsilon\otimes u), \, \, \, &&\mbox{in}\,\, \mathbb{R}^3_+;\\
	&\frac{\p \pi_1}{\p\hat{ n}}=\di (u^\epsilon\otimes u)\cdot\hat{n}, \, \, \, &&\mbox{on}\,\, \{x_3=0\},\\
	&\pi_1(x)\to 0, \,\, \, &&\mbox{as}\,\, |x|\to\infty
	\end{aligned}\right.
	\end{equation}
	and
		\begin{equation}\label{pi2}\left\{
		\begin{aligned}
		&\Delta \pi_2=\di\di (S(D(u)), \, \, \, &&\mbox{in}\,\, \mathbb{R}^3_+;\\
		&\frac{\p \pi_2}{\p\hat{ n}}=\di (S(D(u)) \cdot\hat{n}, \, \, \, &&\mbox{on}\,\, \{x_3=0\},\\
		&\pi_2(x)\to 0, \,\, \, &&\mbox{as}\,\, |x|\to\infty
		\end{aligned}\right.
		\end{equation}
One can simplify that the boundary conditions of equations  both in \eqref{pi1} and \eqref{pi2}. In fact, $\di (u^\epsilon\otimes u)\cdot\hat{n}=u^\epsilon\cdot \nabla u\cdot\hat{n}=0$ as $u_3=0$ on $x_3=0$. By calculation, we observe that on $x_3=0$,
\begin{equation*}
	\di (S(D(u)) \cdot\hat{n}=G'(|D(u)|)\p_3(|D(u)|^2)+G(|D(u)|)\Delta u_3
\end{equation*}
where $G(|D(u)|)=\frac{F'(|D(u)|)}{|D(u)|}.$ From the \eqref{1.4} and incompressibility, we know that  $\Delta u_3=\p_3(|D(u)|^2)=0$ on $x_3=0.$ Therefore, $\pi_1$ and $\pi_2$ can be determined by \eqref{uform} with $f=u^\epsilon\otimes u$ and $f=S(D(u))$ respectively. The estimates \eqref{pre} can be obtained by \eqref{2.9}. The remained proof is the same procedure of Proposition 4.1 in \cite{bu2}.
	\end{proof}

\noindent{\bf The proof of Theorem 
	\ref{mresult}.}	 From \eqref{u4.12}, let $R\to\infty$, we obtain 
\begin{alignat}{3}
&\|u_\epsilon\|_{L^{\infty}(0,T;H)}+2C\int_{0}^{T}\int_{\mathbb{R}^3_+}\left(\phi(D(u_\epsilon))+|u_\epsilon|^r\right)\dif x\dif t\le K,\label{u5.1}
\end{alignat}
for all $r\in [2,\frac{5p}{3}]$ and $C,K$ are independent of $\epsilon.$ From \eqref{s5.2}, we know that 
\begin{equation}\label{s1}\|S(D(u_\epsilon))\|_{L^{s}(Q_T)}\le C(K),~s\in [2,p'].
\end{equation}
 From Proposition 4.1, it shows that
 \begin{equation}\label{5.3}
 \begin{aligned}
 \int_{0}^{T}\|\pi^1_\epsilon\|^{q}_{q}\le C(q,K)~\mbox{and}~\int_{0}^{T}\|\pi^2_\epsilon\|^s_s\le C(s,K)
 \end{aligned}
 \end{equation}
for all $q\in (1,\frac{5p}{6}]$ and  all $s\in [2,p']$ hold.

By \eqref{v3.5}, \eqref{s5.2} and \eqref{5.3}, one obtains 
\begin{equation}\label{5.4}
	\|u_\epsilon^\prime\|_{L^s(0,T;(D^{1,s^\prime}(\mathbb{R}^3_+))^*)\cap L^{q}(0,T;(D^{1,q^\prime}(\mathbb{R}^3_+))^*)}\le K.
\end{equation}

Summing these estimates \eqref{u5.1}, \eqref{s1}, \eqref{5.4} and Aubin-Lions Lemma, we conclude that there exist a sequence of $\{\epsilon_k\}$ with $\epsilon_k\to 0$ as $k\to\infty$,  functions $u\in L^p(0,T; D^{1,\phi}(\mathbb{R}^3_+))\cap L^{\infty}(0,T;H)$ and $\tilde{S}\in L^{s}(Q_T)$, such that as $k\to\infty$
\begin{alignat}{12}
 &u_{\epsilon_k}\stackrel{*}{\rightharpoonup} u, ~~\mbox{weakly}^*\mbox{in}~~L^{\infty}(0,T;H) \label{5.5}\\[3mm]
&u_{\epsilon_k}\rightharpoonup u, ~~\mbox{weakly\,\,in}~~L^p(0,T; D^{1,\phi}(\mathbb{R}^3_+)), \label{5.6}\\[2mm]
&u^\prime_{\epsilon_k}\rightharpoonup u^\prime, ~~\mbox{weakly\,\,in}~~L^s(0,T;(D^{1,s^\prime}(\mathbb{R}^3_+))^*)\cap L^{q}(0,T;(D^{1,q^\prime}(\mathbb{R}^3_+))^*), \label{5.7}\\[2mm]
&S(x,t,D(u_{\epsilon_k}))\rightharpoonup\tilde{S} ~~\mbox{weakly\,\,in}~~L^{s}(Q_T),~s\in [2,p']\label{5.8}\\[2mm]
&u^{\epsilon_k}_{\epsilon_k}\otimes u_{\epsilon_k}\rightharpoonup u\otimes u,~ \mbox{weakly \,\,in}~~L^{q}(Q_T). ~q\in (1,\frac{5p}{6}] \label{5.9}\\[2mm]
& u_{\epsilon_k}\to  u,~ \mbox{strongly \,\,in}~~L^{r}(0,T;L^{r}_{\rm loc}(\mathbb{R}^3_+))~r\in [1,\frac{5p}{3})\label{5.10}\\
&\pi^1_{\epsilon_k} \rightharpoonup \pi^1 ~\mbox{weakly \,\,in}~~L^{q}(Q_T)~q\in (1,\frac{5p}{6}]\label{5.11}\\
&\pi^2_{\epsilon_k} \rightharpoonup \pi^2 ~\mbox{weakly \,\,in}~~L^{s}(Q_T)~q\in [2,p']\label{5.12}
\end{alignat}
Then the identity
\begin{equation}\label{5.13}
 \int_{0}^T \langle u_t,\varphi\rangle\dif t+\int_{Q_T}(\tilde{S}-u\otimes u:D(\varphi))+\pi\nabla\cdot\varphi\dif x\dif t=0
 \end{equation}
 with $\pi=\pi^1+\pi^2$
 holds for any $\varphi\in (L^p(0,T;D_n^{1,\phi}(\mathbb{R}^3_+)) $. 
 
 For convenience, sometimes we denote $u_k=u_{\epsilon_k},\,\,S_k=S(x,t,D(u_{\epsilon_k})).$ To end this proof, we must prove that $\tilde{S}=S(x,t, D(u))$ a.e. in $\mathbb{R}^3_+\times [0,T].$
 As in the \cite {ma}, it suffices to prove that as $k\to\infty$
 \begin{equation}\label{5.14}
 \int_{G\times[\delta,T-\delta]}(S(x,t,D(u_k))-S(x,t,D(u))):(D(u_k)-D(u))\dif x\dif t\to 0
 \end{equation}
 for all bounded compact set $G\subset\overline{\mathbb{R}^3_+}$ and any $0<\delta<\frac{T}{2}.$

 If \eqref{5.14} holds, for any positive cutoff function $\Psi\in C_0^\infty(\overline{\mathbb{R}^3_+}\times (0,T))$, then
\begin{equation*}
\begin{aligned}
 &\int_{Q_T}(\tilde{S}-S(x,t,D(v))):(D(u)-D(v))\Psi\dif x\dif t,\\
 &=\int_{Q_T}(\tilde{S}-S(x,t,D(u_k))):(D(u)-D(v))\Psi\dif x\dif t\\
 &-\int_{Q_T}(S(x,t,D(u_k))-S(x,t,D(u))):(D(u_k)-D(u))\Psi\dif x\dif t\\
 &+\int_{Q_T}(S(x,t,D(u_k))-S(x,t,D(v))):(D(u_k)-D(v))\Psi\dif x\dif t\\
 &-\int_{Q_T}S(D(u))(D(u_k)-D(u))\Psi\dif x\dif t\\
 &:=I_1+I_2+I_3+I_4.
 \end{aligned}
 \end{equation*}
By  \eqref{5.6} and \eqref{5.8}, we know $I_i\to 0 (i=1,2,4)$ as $k\to\infty$. From \eqref{2.4}, it shows that $I_3\ge0.$ By local Minty Trick theorem (see the appendix of \cite{wo}), we know $\tilde{S}=S(x,t, D(u))$ a.e. in $\mathbb{R}^3_+\times (0,T).$

Next, we will prove that \eqref{5.14} by some  $L^\infty-$truncation method. As in \cite{fr} or \cite{bu}, Let $g^k=\phi(\nabla u_k)+\phi(\nabla u)+(|S_k|+|S(x,t,D(u))|)(|D(u_k)+D(u)|).$
The following lemma shows that the properties of $g^k$ on   $\mathbb{R}^3_+\times (0,T)$, its proof can be found in \cite{fr} for the steady case and in \cite{bu} for the unsteady case.
\begin{Lemma}
$\eta>0$, there exists $L\le\frac{\eta}{K}$ and there are a subsequence $\{u_k\}_{k=1}^\infty$ and sets $E^k=\{(x,t)\in\mathbf{R}^3_+\times(0,T): L^2\le|u_k-u|\le L\}$ such that \begin{equation}\label{3.11}
\int_{E^k} g^k\dif x\dif t\le\eta.
\end{equation}
\end{Lemma}
Set $Q^k=\{(x,t)\in\mathbb{R}^3_+\times(0,T): |u_k-u|\le L\}$ and  $\psi^k=(u_k-u)\left(1-\min\left(\frac{|u_k-u|}{L},1\right)\right)$. One can prove that
the following proposition that shows the required properties of $\psi^k$.
\begin{Proposition}\label{phik}

\begin{itemize}
\item[(1)] $\psi^k\in L^p(0,T; D^{1,\phi}(\mathbb{R}^3_+))\cap L^{\infty}(0,T;H)$ and $$\|\psi^k\|_{L^\infty(\mathbb{R}^3_+\times(0,T))}\le L;$$
\item[(2)] $\psi^k\rightharpoonup 0$ in $L^p(0,T;D^{1,\phi}(\mathbb{R}^3_+))$;
\item[(3)] $\psi^k\to 0$ in $ L^s(0,T;L^s_{\rm loc}(\mathbb{R}^3_+))$ for all $1\le s<\infty$;
\item[(4)] $|\di\psi^k|\le \left|\frac{1}{L}(u^k-u)\cdot\nabla|u^k-u|\chi_{Q^k}\right|$; and $$|\di\psi^k|_{L^p(0,T:L^\phi(\mathbb{R}^3_+))}\le C\eta,\,\ |\nabla\psi^k|_{L^p(0,T:L^\phi(\mathbb{R}^3_+))}\le C\eta,$$ where C is independent of $k$
and $\chi_{Q^k}$ denotes the characteristic function of the set $Q^k.$
\end{itemize}
\end{Proposition}
\begin{proof}
It is easy to see that (1) and (2) hold. By the simple calculation, we know that (4) holds. Now we check that (3). Indeed,  since $D^{1,\phi}(\mathbb{R}^3_+)\hookrightarrow L^p(G)$ is compact for all $G\subset\subset \overline{\mathbb{R}^3_+}$, it follows that $\psi^k\to 0$ in $L^p(0,T;L^p(G))$ and there exists a subsequence  $\psi^k \to 0$ a.e. in $G\times (0,T):=G_T.$ Therefore, for $p< s<\infty$ with
$$\int_{G_T}|\psi^k|^s\dif x\dif t\le \|\psi^k\|^{s-p}_{L^\infty(\mathbb{R}^3_+\times(0,T))}\int_{G_T}|\psi^k|^s\dif x\dif t\le L\|\psi^k\|^p_{L^p(G_T)}\le\eta.$$
For $1\le s\le p,$ we have $$\int_{G_T^\delta}|\psi^k|^s\dif x\dif t\le \left(\int_{G_T}|\psi^k|^p\dif x\dif t\right)^{\frac{s}{p}}\left|G_T^\delta\right|^{1-\frac{s}{p}}\le K\left|G_T^\delta\right|^{1-\frac{s}{p}}.$$ Let $\left|G_T^\delta\right|<\delta$ small enough, by Vitali's theorem we have $\psi^k\to 0$ in $ L^s(0,T;L^s_{\rm loc}(\mathbb{R}^3_+)$ for all $1\le s\le p.$ Hence, we proved (3).
\end{proof}
Let
 $$z^k(x,t)=\int_{\mathbb{R}^3_+}N(x,y)\di\psi^k(y,t) \dif y\,\, \forall x\in\mathbb{R}^3_+,$$
 then we have the following estimates
 \begin{equation}\label{3.12}
 \begin{aligned}
 &\|\nabla z^k\|_{L^s(G_T)}\le C(s,G,T)\|\psi^k\|_{L^s(G_T)};\\
 &\|\nabla^2 z^k\|_{L^p(0,T:L^\phi(\mathbb{R}^3_+))}\le C(p)\|\di\psi^k\|_{L^p(0,T:L^\phi(\mathbb{R}^3_+))}\le C\eta.
 \end{aligned}
 \end{equation}
 For  any bounded  domain  with Lipschitz boundary $G\subset\overline{\mathbb{R}^3_+}$ , we can choose another bounded set $G'$ with $G\subset\subset\ G'\subset\subset\overline{\mathbb{R}^3_+}$, and any  positive number $\delta >0.$ Define  smooth functions $\tau\in C^\infty_0(\frac{\delta}{2},T-\frac{\delta}{2})$ and $\zeta\in C^\infty_0(\overline{G'})$ such that $0\le\tau\le1$ in $(\frac{\delta}{2},T-\frac{\delta}{2})$ and $\tau\equiv 1$ in $(\delta,T-\delta),$ $0\le\zeta\le 1$ in $G'$ and $\zeta\equiv 1$ in $G.$
 Let $\varphi^k=\tau\zeta(\psi^k-\nabla z^k)$, then $\varphi^k \in (L^p(0,T;D_n^{1,\phi}(\mathbb{R}^3_+)).$

 Take $\varphi^k$ as a test function in \eqref{5.13} , we follows
 \begin{equation}\label{3.14}
 \begin{aligned}
 &\int_Q S(x,t,D(u_k)):\nabla\varphi^k \dif x\dif t\\
 &=\int_{Q}\tilde{S}:\nabla\varphi^k\dif x\dif t-\int_0^T\langle u'_k-u',\varphi^k\rangle \dif t\\
 &-\int_{Q}\nabla\cdot(u^{\epsilon_k}_k\otimes u_k-u\otimes u)\cdot\varphi^k)\dif x\dif t\\
 &+\int_{Q}(\pi^k-\pi)(\nabla\cdot\varphi^k)\dif x\dif t\\
 &:=D_1+D_2+D_3+D_4.
 \end{aligned}
 \end{equation}

From proposition \ref{phik}, we can know that $D_1\to 0$ as $k\to \infty.$ Since 
\begin{equation}\label{5.15}
\begin{aligned}
\int_{G'_T}\left|\nabla\cdot(u^{\epsilon_k}_k\otimes u_k)\right|^\sigma\dif x\dif t=\int_{G'_T}|(u^{\epsilon_k}_k\cdot\nabla) u_k|^\sigma\dif x\dif t\\
\le \left(\int_{G'_T}|u^{\epsilon_k}_k|^{\sigma\gamma'}\dif x\dif t\right)^\frac{1}{\gamma'}\left(\int_{G'_T}|\nabla u_k|^{\sigma\gamma}\dif x\dif t\right)^\frac{1}{\gamma}.
\end{aligned}
\end{equation}
 From \eqref{u5.1}, let  $\sigma\gamma'=\frac{5p}{3},~\sigma\gamma=p$ for $G'$ is bounded, where $\frac{1}{\gamma}+\frac{1}{\gamma'}=1$, then $\sigma=\frac{5}{8}p>1$ provided that $p>\frac{8}{5}$.  Therefore, there is a constant $C$ independent of $k$ such that
\begin{equation*}
\begin{aligned}
D_3\le\int_{Q_T}\left|\nabla\cdot(u^{\epsilon_k}_k\otimes u_k-u\otimes u)(\psi^k-\nabla z^k)\right|\dif x\dif t\\
\le C\|\psi^k-\nabla z^k\|_{L^{\sigma'}(G'_T)}=o(1) \,\,\, \mbox{as}\,\, k\to\infty.
\end{aligned}
\end{equation*}
 As $\nabla\cdot\varphi^k=\tau\nabla\zeta\cdot(\psi^k-\nabla z^k)$, thus
 $$D_4\le\|\pi^1_k-\pi^1\|_{L^q(G'_T)}\|\nabla\cdot \varphi^k\|_{L^{q'}(G'_T)}+\|\pi^2_k-\pi^2\|_{L^s(G'_T)}\|\nabla\cdot \varphi^k\|_{L^{s'}(G'_T)}.$$  thus from Proposition \ref{phik} (3), and \eqref{3.12}, it shows that  $$\|\nabla\cdot \varphi^k\|_{L^{r'}(G'_T)}=\|\tau\nabla\zeta\cdot(\psi^k-\nabla z^k)\|_{L^{r'}(G'_T)}\le C\|\psi^k-\nabla z^k\|_{L^{r'}(G'_T)}\to 0\,\,\mbox{as}\,\,k\to\infty.$$ By the bounded estimates for $\pi^i_k(i=1,2)$ in \eqref{5.3}, we claim that
 $$D_4=o(1)\,\,\mbox{as}\,\,k\to\infty.$$

To estimate the term $D_2$, we set $w=u_k-u\in L^p(0,T;D^{1,\phi}(\mathbb{R}^3_+)).$ As  the argument of the footnotes on page 79 in \cite {bu},  for $1<p<6$, there exists a sequence $\{w_n\}\subset C^\infty(0,T;\mathcal{D}(\mathbb{R}_+^3))$ with $\di w_n=0$, satisfying $w_n'\to w'$ in $L^{p'}(0,T;W^{-1,p'}(G'))$  and $w_n\to w$ in $L^p(0,T;D^{1,\phi}(\mathbb{R}^3_+))$. One has
\begin{equation*}
\begin{aligned}
&D_2=\int_0^T\langle (\tau u_k)'-(\tau u)',\zeta(\psi^k-\nabla z^k)\rangle \dif t\\
&=\lim_{n\to\infty}\int_0^T\langle (w_n',\tau\zeta(w_n(1-\min(\frac{|w_n|}{L},1))-\nabla z^k)\rangle \dif t=D_{2_1}+D_{2_2}.
\end{aligned}
\end{equation*}
We can refer \cite{bu}, then it is easy to obtain that 
$D_{2_1}\le C |G'_T|L^2.$

Since $\di w_n=0$, from \eqref{3.12} and the  estimates \eqref{5.4} of  $u_k'$, we can obtain that $D_{2_2}=o(1).$

So far, we can conclude that
$$\int^T_0\int_{\mathbb{R}^3_+}S(x,t,D(u_k)):D(\varphi^k)\dif x\dif t\le o(1)+C\eta,\,\,\mbox{as}\,\, k\to\infty.$$
It follows that
\begin{equation*}
\begin{aligned}
&\int^T_0\int_{\mathbb{R}^3_+}S(x,t,D(u_k)):D(\psi^k)\tau\zeta\dif x\dif t\\
&\le \int^T_0\int_{\mathbb{R}^3_+}S(x,t,D(u_k)):D(\nabla z^k)\tau\zeta\dif x\dif t\\
&-\int^T_0\int_{\mathbb{R}^3_+}S(x,t,D(u_k)):({\rm sym}(\psi^k-\nabla z^k)\otimes \nabla\zeta)\tau\dif x\dif t+o(1)+C\eta.
\end{aligned}
\end{equation*}

From \eqref{3.12} $$\left|\int^T_0\int_{\mathbb{R}^3_+}S(x,t,D(u_k)):D(\nabla z^k)\tau\zeta\dif x\dif t\right|\le K\|\nabla^2 z^k\|_{L^p(Q)}\le C\eta,$$ and $$\left|-\int^T_0\int_{\mathbb{R}^3_+}S(x,t,D(u_k)):((\psi^k-\nabla z^k)\otimes \nabla\zeta)\tau\dif x\dif t\right|\le C K\|\psi^k-\nabla z^k\|_{L^p(G'_T)}=o(1). $$
Therefore, we have
$$\int^T_0\int_{\mathbb{R}^3_+}S(x,t,D(u_k)):D(\psi^k)\tau\zeta\dif x\dif t\le C\eta+o(1)\,\,\mbox{as}\,\,k\to\infty,$$ where $C$ is independent of $k,\eta.$

On the other hand,
\begin{equation*}
\begin{aligned}
&\int^T_0\int_{\mathbb{R}^3_+}S(x,t,D(u_k)):D(\psi^k)\tau\zeta\dif x\dif t\\
&= \int_{Q^k}S(x,t,D(u_k)):D(u_k-u)\left(1-\min\left(\frac{|u_k-u|}{L},1\right)\right)\tau\zeta\dif x\dif t\\
&+\int_{Q^k}S(x,t,D(u_k)):{\rm sym}\left((\frac{u_k-u}{L})\otimes \nabla|u_k-u|\right)\tau\eta\dif x\dif t\\
&=\int_{Q^k}(S(x,t,D(u_k))-S(x,t,D(u))):D(u_k-u)\tau\zeta\dif x\dif t\\
&+\int_{Q^k}S(x,t,D(u)):D(u_k-u)\tau\zeta\dif x\dif t\\
&-\int_{Q^k}S(x,t,D(u_k)):D(u_k-u)\min\left(\frac{|u_k-u|}{L},1\right)\tau\zeta\dif x\dif t\\
&+\int_{Q^k}S(x,t,D(u_k)):{\rm sym}\left((\frac{u_k-u}{L})\otimes \nabla|u_k-u|\right)\tau\eta\dif x\dif t\\
&:=J_1+J_2+J_3+J_4.
\end{aligned}
\end{equation*}
  Clearly, since $D(u_k)\rightharpoonup D(u)$ in $L^p(0,T;L^\phi(\mathbb{R}^3_+))$, thus $J_2\to 0$ as $k\to\infty.$ As in (4) of Proposition \ref{phik}, we can compute that
\begin{eqnarray*}
&|J_3|+|J_4|&\le\int_{E^k}|S(x,t,D(u_k)):D(u_k-u)|\dif x\dif t\\
&&+L\int_{Q^k\setminus E^k}|S(x,t,D(u_k)):D(u_k-u)|\dif x\dif t\\
&&\le \int_{E^k}g^k\dif x\dif t+L\int_{Q^k\setminus E^k}g^k\dif x\dif t\le\eta+KL\\
&&\le C\eta.
\end{eqnarray*}
Consequently,
\begin{equation}\label{3.15}
\int_{Q^k}(S(x,t,D(u_k))-S(x,t,D(u))):D(u_k-u)\tau\zeta\dif x\dif t\le o(1)+C\eta,\,\,\mbox{as}\,\,k\to\infty.
\end{equation}
As above, we know that $u_k\to u$ a.e. in $Q_T$.  Hence, choose a subsequence which still denote $\{u_k\}$ such that $|G'_T\setminus Q^k|\le 2^{-k},$ for all $k\in\mathbf{N}.$
Thus there exists $k_0\in\mathbf{N}$ such that $2^{-k_0}<\eta$, and one can find $$\sum_{k=k_0+1}^{\infty}|G'_T\setminus Q^k|\le 2^{-k_0}<\eta.$$
Setting $M=\bigcup_{k=k_0+1}^{\infty}(G'_T\setminus Q^k),$  we easily obtain
\begin{equation}\label{3.16}
\int_{M}(S(x,t,D(u_k))-S(x,t,D(u))):D(u_k-u)\tau\zeta\dif x\dif t\le C\eta.
\end{equation}
Whence \eqref{3.15} and \eqref{3.16} infer that
\begin{equation*}
\int_{G'_T}(S(x,t,D(u_k))-S(x,t,D(u))):D(u_k-u)\tau\zeta\dif x\dif t\le o(1)+C\eta,\,\,\mbox{as}\,\,k\to\infty.
\end{equation*}
From \eqref{2.4} it implies $$\int_{G_{\delta T}}(S(x,t,D(u_k))-S(x,t,D(u))):D(u_k-u)\dif x\dif t\to 0,\,\,\mbox{as}\,\,k\to\infty.$$ This proves the main theorem.
\begin{Remark}
We can assume that $f\in L^{p'}(0,T; (D^{1,\phi}(\mathbb{R}^3_+)^*)\cap L^2(Q_T)$. Then all estimates \eqref{u5.1}-\eqref{5.4} also depend on $\|f\|_{L^{p'}(0,T; (D^{1,\phi}(\mathbb{R}^3_+)^*)\cap L^2(Q_T)}$. In the proof of \eqref{3.14}, we must estimate the term $\int_{0}^T\langle f,\varphi^k\rangle \dif t.$ Indeed it is easy to obtain it as follows;
\begin{eqnarray*}
&\int_{0}^T\langle f,\varphi^k\rangle \dif t&\le \int_{0}^T\langle f,\tau\zeta(u_k-u)\rangle \dif t
-\int_{0}^T\langle f,\tau\zeta(u_k-u)\min\left(\frac{u_k-u}{L},1\right)\rangle \dif t\\
&&-\int_{0}^T\langle f,\tau\zeta\nabla z^k\rangle \dif t.
\end{eqnarray*}
The first term on the right vanishes as $k\to\infty$, while the second term is estimated analogously as $J_3$ and $J_4$. Finally, the third term is small thanks to \eqref{3.12}. Therefore, we obtain
$$\int_{0}^T\langle f,\varphi^k\rangle \dif t\le o(1)+c\eta.$$
It shows that this term can not change the statement of the main theorem.
\end{Remark}
\section{Conclusion}
In this work, we focus on studying  the existence of  weak solutions for  partial differential equations of fluids with shear thinning dependent viscosities. It is easy to obtain this issue about  fluids with shear thick case without Orlicz space frame. We have generalized the previous results in the case  of unbounded domain until $p>\frac{8}{5}.$

\noindent{\bf Acknowledgements}

 The work of author is a part of Projects 11571279,and 11201411 supported by National Natural Science Foundation of China  and a part of Project GJJ151036 supported by Education Department of Jiangxi Province and  partly supported by Youth Innovation Group of Applied Mathematics  in Yichun University(2012TD006).

\end{document}